\documentclass[11pt,reqno]{amsart}
\usepackage{latexsym,amssymb,amsfonts,amsmath,graphicx,fullpage,url,color,tikz}
\usepackage[utf8]{inputenc}
\usepackage[vcentermath,enableskew]{youngtab}
\usepackage[all]{xy}
\usepackage{verbatim}
\usepackage{ytableau}
\usepackage{makecell}
\usepackage[export]{adjustbox}
\usepackage{multicol}
\usepackage{cases}

\usepackage{rotating}
\usepackage{tikz,graphicx}
\tikzstyle{vertex}=[circle, draw, inner sep=0pt, minimum size=4pt]

\allowdisplaybreaks

\newtheorem{theorem}{Theorem}[section]

\newtheorem{lemma}[theorem]{Lemma}
\newtheorem{corollary}[theorem]{Corollary}
\newtheorem*{MainTheorem1}{Theorem~\ref{thm:ineq_desc}}
\newtheorem*{MainTheorem2}{Theorem~\ref{thm:face}}
\newtheorem*{MainTheorem3}{Theorem~\ref{thm:int_equiv}}
\newtheorem*{MainCorollary1}{Corollary~\ref{cor:volume}}

\theoremstyle{definition}
\newtheorem{definition}[theorem]{Definition}
\newtheorem{example}[theorem]{Example}

\theoremstyle{remark}
\newtheorem{remark}[theorem]{Remark}
\numberwithin{equation}{section}
\usepackage{multicol}

\renewcommand{\P}{\mathcal{P}}

\newcommand{\pasm}{\mathrm{PASM}}
\newcommand{\asm}{\mathrm{ASM}}
\newcommand{\Cat}{\mathrm{Cat}}
\newcommand{\asmcry}{\mathrm{ASMCRY}}

\newcommand{\aff}{\mathrm{aff}}

\author{Dylan Heuer, Sara Solhjem and Jessica Striker}
\email{heuerd@msoe.edu,sara.solhjem@mnstate.edu jessica.striker@ndsu.edu}

\address{Milwaukee School of Engineering, Minnesota State University Moorhead, North Dakota State University}
\title{ On {\Large $\nu$}  faces of partial alternating sign matrix polytopes} 

\begin{document}
\begin{abstract}
We define and study the $(\nu / \lambda)$-partial alternating sign matrix polytope,  motivated by connections to the Chan-Robbins-Yuen polytope and the $\nu$-Tamari lattice.
We determine the inequality description and show this polytope is a face of the partial alternating sign matrix polytope of [Heuer, Striker 2022]. We show that the $(\nu / \lambda)$-partial ASM polytope is an order polytope and a flow polytope. 
\end{abstract}

\maketitle

\section{Introduction and background}
The ASM-CRY polytope is an alternating sign matrix analogue of the CRY polytope: a face of the Birkhoff polytope studied by Chan, Robbins, and Yuen~\cite{CRY2000} whose volume is given as a product of Catalan numbers~\cite{ZeilbergerCRY}. 
It was shown in \cite{ASMCRY} that the ASM-CRY polytope is integrally equivalent to an order polytope, which provided a proof of its lovely combinatorial properties, such as counting formulas for its volume and vertices. More generally, \cite{ASMCRY} showed that order polytopes of strongly planar posets are flow polytopes. Subsequent work \cite{BGHHKMY2019,MeszarosMorales2019,LiuMeszarosStDizier2019, JangKim2020,unifyingframework} built upon this connection between order and flow polytopes, but there was no interpretation in terms of alternating sign matrices. 
In this paper, we reunite these topics.
In particular, we were inspired by the polytope perspective on the $\nu$-Tamari lattice in \cite{unifyingframework} to define and study an analogue of the ASM-CRY polytope of \cite{ASMCRY} using the partial alternating sign matrix polytope of \cite{HeuerStriker}. Furthermore, we show this polytope is both an order polytope and a flow polytope.

We first give the definition of alternating sign matrices and  discuss relevant prior work on the ASM-CRY polytope; we then describe our main results.

\begin{definition}
\label{def:asm}
An \emph{alternating sign matrix} is an $n \times n$ matrix $M = \left(M_{ij}\right)$ with entries in $\left\{-1,0,1\right\}$ such that:
\begin{align}
\label{eq:ASM1}
\displaystyle\sum_{i'=1}^{i} M_{i'j} &\in \left\{0,1\right\}, & \mbox{ for all } 1 \leq i, j \leq n \\
\label{eq:ASM2}
\displaystyle\sum_{j'=1}^{j} M_{ij'} &\in \left\{0,1\right\}, & \mbox{ for all } 1 \leq i, j \leq n \\
\label{eq:ASM3}
\displaystyle\sum_{i'=1}^{n} M_{i'j} &= \displaystyle\sum_{j'=1}^{n} M_{ij'} = 1.
\end{align}
We denote the set of all $n \times n$ alternating sign matrices as $\asm_{n}$. Define the \emph{alternating sign matrix polytope} $\asm(n)$   as the convex hull of these matrices.
\end{definition}

Alternating sign matrices arose in the study of the lambda-determinant~\cite{RobbinsRumsey}. They have a nice enumeration formula~\cite{MRRASMDPP,ZEILASM,kuperbergASMpf} and connections to statistical physics~\cite{kuperbergASMpf}.
The alternating sign matrix polytope was defined  as above by Striker in~\cite{ASMPoly}, where she also found an inequality description, enumerated the facets, and studied its projection to the permutohedron. Independently, Behrend and Knight defined this polytope in terms of inequalities and proved the vertex description~\cite{Behrend2007HigherSA}. In addition, they studied lattice points in the $r$th-dilate of the alternating sign matrix
polytope, which they called \emph{higher spin} alternating sign matrices.

In \cite{ASMCRY}, a face of the alternating sign matrix polytope with some prescribed zeros in certain parts of the matrix was studied and shown to be integrally equivalent to an order polytope. We give the definition of this polytope and state the relevant theorems below. (Note our convention differs from that of \cite{ASMCRY} by vertical reflection of the matrix.)

Let $\delta_n = (n-1, n-2,\ldots , 2, 1)$ be the staircase partition and 
consider the partition $\lambda = (\lambda_1, \lambda_2,\ldots, \lambda_{\ell}) \subseteq \delta_n$ as the positions $(i,j)$ of an $n\times n$ matrix given by $\{(i, j) \ | \ 1 \leq i \leq \ell, 1\leq j\leq \lambda_i\}$. 

\begin{definition}
\label{def:asmcry}
Define the \emph{$\lambda$-ASMCRY polytope}
$\asmcry(\lambda,n) := 
\{(a_{ij})^n_{i,j=1}\in \asm(n) \  | \ a_{ij} = 0  \text{  for } i + j \geq n+3 \text{ or } (i, j) \in\lambda$\}. 
\end{definition}

The main results about this polytope from \cite{ASMCRY} are the following; see Section~\ref{sec:orderflow} for definitions of the notation.
\begin{theorem}[\protect{\cite[Theorem 1.1]{ASMCRY}}]
\label{thm:asmcryequiv}
The polytope $\asmcry(\lambda,n)$ is integrally equivalent to the order polytope $\mathcal{O}(P(\delta_n / \lambda))$
and the flow polytope $\mathcal{F}_{G_{P(\delta_n / \lambda)}}$.
\end{theorem}

\begin{corollary}[\protect{\cite[Corollaries 5.7 and 1.2]{ASMCRY}}]
The normalized volume of $\asmcry(\lambda,n)$  is given by the number of linear extensions
$e(P(\delta_n / \lambda))$. Its Ehrhart polynomial $L_{\asmcry(\lambda,n)}(t)$ equals the order polynomial
$\Omega(P(\delta_n / \lambda),t + 1)$.

Moreover, $\asmcry(\emptyset, n)$ has $\Cat(n) = \frac{1}{n+1}\binom{2n}{n}$ vertices, its normalized volume is given by the number of standard Young tableaux of shape $\delta_n$, and its Ehrhart polynomial is $L_{\asmcry(\emptyset, n)}(t) = \Omega(P(\delta_n),t + 1) = \prod_{1\leq i<j\leq n} \frac{2t + i + j - 1}{i + j - 1}$.
\end{corollary}

The partial alternating sign matrix polytope $\pasm(m,n)$ was introduced and studied in \cite{HeuerStriker}. This polytope is defined as the convex hull of \emph{partial alternating sign matrices}, a natural extension of alternating sign matrices, whose definition we give in Definition~\ref{def:pasm}. The authors of \cite{HeuerStriker} proved an inequality description of the partial alternating sign matrix polytope. They also enumerated the facets and characterized the face lattice, in an analogous way to that of the alternating sign matrix polytope. 

In this paper, we study polytopes $\pasm(\nu/\lambda,m,n)$ (see Definition~\ref{def:poly_nu_lambda}) contained in partial alternating sign matrix polytopes, with certain entries set equal to $0$, analogous to the ASM-CRY polytope of Definition~\ref{def:asmcry}. We find similar connections to order and flow polytopes and the results of \cite{ASMCRY} discussed above. Our theorems are not straightforward generalizations of results of either \cite{HeuerStriker} or \cite{ASMCRY}. For instance, our proof of Theorem~\ref{thm:ineq_desc} involves a novel induction. Also, we show in Proposition~\ref{prop:equiv} that when restricted to square matrices and the appropriate partition shapes, $\pasm(\delta_n/\lambda,n,n)$ is not equal to $\asmcry(\lambda, n)$, but the polytopes are integrally equivalent.

Our first main result appears in Section~\ref{sec:ineq}, where we prove the following inequality description.
\begin{MainTheorem1}
Let $\lambda,\nu$ be partitions inside the partition $n^m$ such that $\lambda\subseteq\nu$. Then $\pasm(\nu/\lambda,m,n)$ is given by the inequalities:
\begin{align*}
0  \leq \displaystyle\sum_{i'=1}^{i} X_{i'j} \leq 1, &\text{ for all } 1 \leq i \leq m, 1 \leq j \leq n \\
0  \leq \displaystyle\sum_{j'=1}^{j} X_{ij'} \leq 1, &\text{ for all }  1 \leq i \leq m, 1 \leq j \leq n \\
\displaystyle\sum_{i=1}^{m} X_{i1} = 1 & \text{ and } \displaystyle\sum_{j=1}^{n} X_{1j} = 1 \\
\displaystyle\sum_{j=1}^{n} X_{ij} = 0 &\text{ for } 2 \leq i \leq m \text{ and } \sum_{i=1}^{m}{1 \leq i \leq m} X_{ij} = 0 \text{ for } 2 \leq j \leq n \\
X_{ij}  = 0 & \text{ if } j\leq \lambda_{i}  \text{ for } 1\leq i\leq n \\
X_{ij}  = 0 & \text{ if } \nu_{i-1} < j \leq n \text{ for } i>1, \text{ and } X_{ij} = 0 \text{ if } \nu_1 +1 < j  \leq n.
\end{align*}
\end{MainTheorem1}

Then, in Section~\ref{sec:face}, we show that $\pasm(\nu / \lambda,m,n)$ is a face of $\pasm(m,n)$.

\begin{MainTheorem2}
Let $\lambda \subseteq \nu \subseteq (n-1)^{m-1}$. Then $\pasm(\nu / \lambda,m,n)$ is a face of $\pasm(m,n)$ of dimension $|\nu|-|\lambda|$.
\end{MainTheorem2}

Our final main result, in Section~\ref{sec:orderflow}, shows that $\pasm(\nu / \lambda,m,n)$ is integrally equivalent to an order polytope of a strongly planar graph (and thus a flow polytope, see Corollary~\ref{cor:flow}).

\begin{MainTheorem3}
The $(\nu / \lambda)$-partial ASM polytope $\pasm(\nu / \lambda,m,n)$ is integrally equivalent to the order polytope $\mathcal{O}(P(\nu / \lambda))$.
\end{MainTheorem3}

This leads to the following result on the volume and Ehrhart polynomial of $\pasm(\nu / \lambda,m,n)$. Note for a poset $P$, $\Omega(P,t)$ denotes the number of order-preserving maps $\eta: P \rightarrow \{1, \ldots, t\}$.

\begin{MainCorollary1}
The normalized volume of $\pasm(\nu / \lambda,m,n)$ is  $e(P(\nu / \lambda))$  and its Ehrhart polynomial is $L_{\pasm(\nu / \lambda,m,n)}(t) = \Omega(P(\nu / \lambda),t + 1)$. 
In particular, $e(P(\nu / \lambda))$ is the number of skew standard Young tableaux, enumerated by the Naruse hook-length formula
\[e(P(\nu / \lambda))=|\nu / \lambda|!\sum_{D\in\mathcal{E}(\nu / \lambda)}\prod_{u\in[\nu]\setminus D}\frac{1}{h(u)},\]
where $\mathcal{E}(\nu / \lambda)$ is the set of excited diagrams of $\nu / \lambda$ and $h(u)$ is the hook number of a box $u$.
\end{MainCorollary1}

\section{Definition of the $(\nu / \lambda)$-partial ASM polytope}
In this section, we define matrices associated to partitions and use them to define the polytope $\pasm(\nu / \lambda,m,n)$ studied in this paper. We also prove Lemma~\ref{lem:Mmu}, which is used in the next section in proving the inequality description of this polytope.

A \emph{partition} $\mu=(\mu_1,\mu_2,\ldots)$ is an infinite sequence of integers satisfying $\mu_1\geq \mu_2\geq \cdots \geq \mu_{\ell}>\mu_{\ell+1}=0$ for some $\ell>0$. Call $\ell$ the \emph{length} of the partition.

\begin{definition}[\protect{\cite[Def.\ 2.5]{HeuerStriker}}]
\label{def:pasm}
An $m \times n$ \emph{partial alternating sign matrix} is an $m \times n$ matrix $M = \left(M_{ij}\right)$ with entries in $\left\{-1,0,1\right\}$ such that:
\begin{align}
\label{eq:pasm1}
\displaystyle\sum_{i'=1}^{i} M_{i'j} &\in \left\{0,1\right\}, & \mbox{ for all } 1 \leq i \leq m , 1 \leq j \leq n. \\
\label{eq:pasm2}
\displaystyle\sum_{j'=1}^{j} M_{ij'} &\in \left\{0,1\right\}, & \mbox{ for all } 1 \leq i \leq m, 1 \leq j \leq n.
\end{align}
We denote the set of all $m \times n$ partial alternating sign matrices as $\pasm_{m,n}$.
\end{definition}


The polytope determined as the convex hull of these matrices is defined below. See \cite{HeuerStriker} for many more facts about this polytope, including its inequality description, vertices, facet count, and face lattice.
\begin{definition}[\protect{\cite[Def.\ 4.1]{HeuerStriker}}]
Define the $(m,n)$-partial alternating sign matrix polytope $\pasm(m,n)$ to be the convex hull (as vectors in $\mathbb{R}^{mn}$) of all matrices in $\pasm_{m,n}$. 
\end{definition}

Given a partition contained in the rectangular partition $n^m$, we specify a partial sum matrix. We then determine the row and column sums of all such matrices.
\begin{definition}
\label{Mmu}
 Suppose $\mu$ is a partition with $\ell<m$ positive parts and $\mu_1<n$.
Let $M^{\mu}(m, n)$  be the $m \times n$ matrix defined entrywise as below, where $k$ satisfies $1\leq k\leq m-1$:

\begin{subnumcases}{M^{\mu}_{ij}(m,n)=}
1 & \text{if } $(i,j)=(1,\mu_1+1)$  \label{case0}\\
1 & \text{if } $(i,j)=(k+1,\mu_{k+1}+1)$ \text{ and } $\mu_k>\mu_{k+1}$ \label{case1}\\
-1 & \text{if } $(i,j)=(k+1,\mu_{k}+1)$ \text{ and } $\mu_k>\mu_{k+1}$  \label{case2}\\
0 & \text{otherwise.}\label{case4}
\end{subnumcases}

When $m,n$ are understood, we use the notation $M^{\mu}$.
\end{definition}

\begin{lemma}
\label{lem:Mmu}
Fix $m,n>0$. For any partition $\mu \subseteq n^m$, $M^{\mu}(m,n)$ is a partial alternating sign matrix with first row sum $1$, first column sum $1$ and all other row and column sums $0$.
\end{lemma}
\begin{proof}
By construction, $M^{\mu}(m,n)$ is an $m\times n$ matrix with entries in $\{-1,0,1\}$. We need to show (\ref{eq:pasm1}) and (\ref{eq:pasm2}).

Also by construction, there is at most one $1$ in any row or column, so the partial sums cannot exceed~$1$. Furthermore, any $-1$, of which there is at most one in any row or column, appears to the right of a $1$ in the same row  and below another $1$ in the same column. We justify this as follows. 

Consider row $1$ of $M^{\mu}$. Since (\ref{case0}) is the only case that applies, row $1$ has at most one $1$ and no $-1$. This satisfies (\ref{eq:pasm2}) for $i=1$. Suppose $\mu_k=\mu_{k+1}$, then all entries of row $k+1$ in $M^{\mu}$ are $0$, satisfying (\ref{eq:pasm2}).
Suppose $\mu_k>\mu_{k+1}$. Then by (\ref{case1}), $M^{\mu}_{k+1,\mu_{k+1}+1}=1$, by (\ref{case2}) $M^{\mu}_{k+1,\mu_{k}+1}=-1$ if $\mu_k<n$, and all other entries of row $k+1$ equal $0$. Thus, row $k+1$ consists of a $1$ that has  a $-1$ somewhere later in its row and all other entries equal to $0$; this satisfies (\ref{eq:pasm2}) for $i=k+1$. 

We now show (\ref{eq:pasm1}) holds. Consider column $1$ of $M^{\mu}$. The only case that applies to column $1$ is  (\ref{case1}) when $k=\ell$, thus column $1$ has exactly one $1$ and no $-1$. This satisfies (\ref{eq:pasm1}) for $j=1$. For column $j>1$, suppose $M^{\mu}_{ij}=-1$. Then by (\ref{case2}), $i=k+1$ and $j=\mu_k+1$ for some $k$ such that $\mu_k>\mu_{k+1}$. Suppose $p$ is the largest integer less than $k$ such that $\mu_p>\mu_k$. So $\mu_{p+1}=\mu_{k}$. So by (\ref{case1}), $M^{\mu}_{p+1,\mu_{p+1}+1}=M^{\mu}_{p+1,\mu_k+1}=1$, so the $-1$ is below a $1$ and all other entries of column $\mu_k+1$ equal $0$. This satisfies (\ref{eq:pasm1}) for $j>1$.

This means that a $1$ is always ``before'' a $-1$ so the partial sums are never negative.
\end{proof}

See Figure~\ref{fig:5x7example} for an example.

\begin{figure}[htbp]
\begin{tikzpicture}
\begin{scope}[scale=0.5]
    \draw[fill, color=gray!33] (0,0) rectangle (2,1);
    \draw[fill, color=gray!33] (1,1) rectangle (4,2);
    \draw[fill, color=gray!33] (3,2) rectangle (4,3);
     \draw[fill, color=gray!33] (3,3) rectangle (6,4);
      \draw[fill, color=gray!33] (5,4) rectangle (6,5);
    \draw[fill, color=red!33] (2,0) rectangle (4,1);
    \draw[fill, color=red!33] (4,0) rectangle (7,3);
    \draw[fill, color=red!33] (6,3) rectangle (7,5);
    \draw[very thin, color=gray!100] (0,0) grid (7,5);
    \draw[very thick, blue] (0,0) -- (0,1) -- (1,1) -- (1,2) -- (3,2) -- (3,4) -- (5,4) -- (5,5) -- (7,5);
    \node at (0.5,4.5) {0};
    \node at (1.5,4.5) {0};
    \node at (2.5,4.5) {0};
    \node at (3.5,4.5) {0};
    \node at (4.5,4.5) {0};
    \node at (5.5,4.5) {1};
    \node at (6.5,4.5) {0};
    
    \node at (0.5,3.5) {0};
    \node at (1.5,3.5) {0};
    \node at (2.5,3.5) {0};
    \node at (3.5,3.5) {1};
    \node at (4.5,3.5) {0};
    \node at (5.5,3.5) {-1};
    \node at (6.5,3.5) {0};
    
    \node at (0.5,2.5) {0};
    \node at (1.5,2.5) {0};
    \node at (2.5,2.5) {0};
    \node at (3.5,2.5) {0};
    \node at (4.5,2.5) {0};
    \node at (5.5,2.5) {0};
    \node at (6.5,2.5) {0};
    
    \node at (0.5,1.5) {0};
    \node at (1.5,1.5) {1};
    \node at (2.5,1.5) {0};
    \node at (3.5,1.5) {-1};
    \node at (4.5,1.5) {0};
    \node at (5.5,1.5) {0};
    \node at (6.5,1.5) {0};
    
    \node at (0.5,0.5) {1};
    \node at (1.5,0.5) {-1};
    \node at (2.5,0.5) {0};
    \node at (3.5,0.5) {0};
    \node at (4.5,0.5) {0};
    \node at (5.5,0.5) {0};
    \node at (6.5,0.5) {0};
    
\end{scope}
\end{tikzpicture}
\caption{The matrix $M^{(5,3,3,1)}_{5,7}$ with $(5,3,3,1)$ shown in white. The border strip (Definition~\ref{def:border_strip}) associated to $(5,3,3,1)$ is shown in grey.}
\label{fig:5x7example}
\end{figure}

We now define the set of matrices of the associated polytope that will be the main focus of study in this paper.
\begin{definition}
Let $\lambda,\mu,\nu$ be partitions inside the partition $(n-1)^{m-1}$ such that $\lambda\subseteq\mu\subseteq\nu$.
Define $\pasm_{m,n}^{\nu / \lambda} = \left\{M^{\mu}(m,n)\; | \; \lambda\subseteq\mu\subseteq \nu \right\}$. See Figure \ref{fig:4x5examples}.
\end{definition}

\begin{definition}
\label{def:poly_nu_lambda}
Define the $(\nu / \lambda)$-partial ASM polytope $\pasm(\nu / \lambda,m,n)$ to be the convex hull (as vectors in $\mathbb{R}^{mn}$) of all matrices in  $\pasm^{\nu / \lambda}_{m,n}$. 
\end{definition}

\begin{remark}
\label{remark:pasmmn}
Since $\pasm(\nu / \lambda,m,n)$ is constructed as the convex hull of a subset of $m\times n$ partial alternating sign matrices, it is contained in the convex hull of all $m\times n$ partial alternating sign matrices; this is the polytope $\pasm(m,n)$.
\end{remark}

\section{Inequality description}
\label{sec:ineq}
In this section, we prove an inequality description for the polytope $\pasm(\nu/\lambda,m,n)$.

We use the following definition in Theorem~\ref{thm:ineq_desc}.
\begin{definition}
\label{def:border_strip}
A skew partition $\omega / \nu$ is a \emph{border strip} if it is connected  and contains no $2\times 2$ square. We say $\omega / \nu$ is the border strip associated to $\nu$ if $\omega / \nu$ is the smallest border strip containing all squares that are one square either south or east of blocks in $\nu$. See the gray boxes in Figure \ref{fig:5x7example}.
\end{definition}

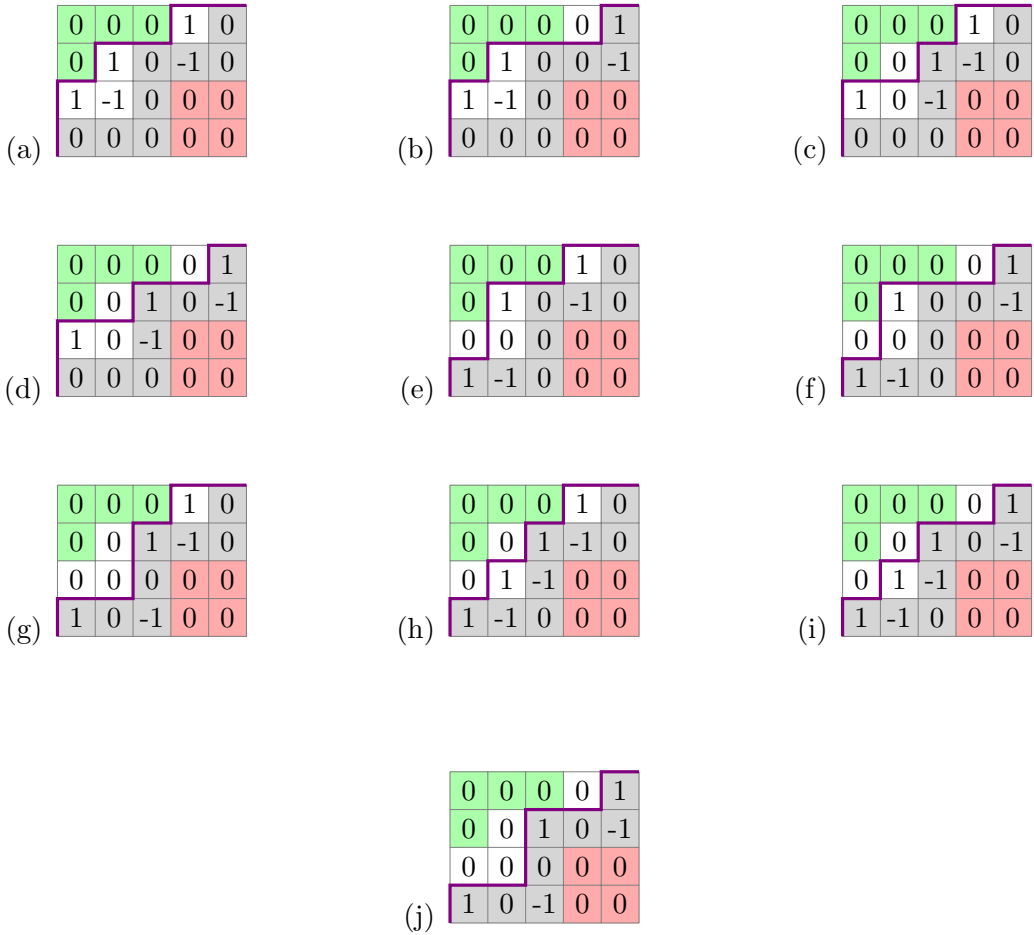
\begin{figure}[ht!]
\begin{enumerate}
\begin{multicols}{3}

\item[(a)]

\begin{tikzpicture}
\begin{scope}[scale=0.5]
    \draw[fill, color=green!33] (0,2) rectangle (1,4);
    \draw[fill, color=green!33] (1,3) rectangle (3,4);
   
      \draw[fill, color=gray!33] (0,0) rectangle (2,1);
    \draw[fill, color=gray!33] (2,0) rectangle (3,3);
    \draw[fill, color=gray!33] (3,2) rectangle (5,3);
    \draw[fill, color=gray!33] (4,3) rectangle (5,4);
     \draw[fill, color=red!33] (3,0) rectangle (5,2);
 
    \draw[very thin, color=gray!100] (0,0) grid (5,4);
    \draw[very thick, violet] (0,0) -- (0,2) -- (1,2) -- (1,3) -- (3,3) -- (3,4) -- (5,4) ;

    \node at (0.5,3.5) {0};
    \node at (1.5,3.5) {0};
    \node at (2.5,3.5) {0};
    \node at (3.5,3.5) {1};
    \node at (4.5,3.5) {0};

    \node at (0.5,2.5) {0};
    \node at (1.5,2.5) {1};
    \node at (2.5,2.5) {0};
    \node at (3.5,2.5) {-1};
    \node at (4.5,2.5) {0};

    \node at (0.5,1.5) {1};
    \node at (1.5,1.5) {-1};
    \node at (2.5,1.5) {0};
    \node at (3.5,1.5) {0};
    \node at (4.5,1.5) {0};

    \node at (0.5,0.5) {0};
    \node at (1.5,0.5) {0};
    \node at (2.5,0.5) {0};
    \node at (3.5,0.5) {0};
    \node at (4.5,0.5) {0};

\end{scope}
\end{tikzpicture}
\vspace{1cm}

\item[(d)]
\begin{tikzpicture}
\begin{scope}[scale=0.5]
    \draw[fill, color=green!33] (0,2) rectangle (1,4);
    \draw[fill, color=green!33] (1,3) rectangle (3,4);
   
    \draw[fill, color=gray!33] (0,0) rectangle (2,1);
    \draw[fill, color=gray!33] (2,0) rectangle (3,3);
    \draw[fill, color=gray!33] (3,2) rectangle (5,3);
    \draw[fill, color=gray!33] (4,3) rectangle (5,4);
     \draw[fill, color=red!33] (3,0) rectangle (5,2);
 
    \draw[very thin, color=gray!100] (0,0) grid (5,4);
    \draw[very thick, violet] (0,0) -- (0,2) -- (2,2) -- (2,3) -- (4,3) -- (4,4) -- (5,4) ;

    \node at (0.5,3.5) {0};
    \node at (1.5,3.5) {0};
    \node at (2.5,3.5) {0};
    \node at (3.5,3.5) {0};
    \node at (4.5,3.5) {1};

    \node at (0.5,2.5) {0};
    \node at (1.5,2.5) {0};
    \node at (2.5,2.5) {1};
    \node at (3.5,2.5) {0};
    \node at (4.5,2.5) {-1};

    \node at (0.5,1.5) {1};
    \node at (1.5,1.5) {0};
    \node at (2.5,1.5) {-1};
    \node at (3.5,1.5) {0};
    \node at (4.5,1.5) {0};

    \node at (0.5,0.5) {0};
    \node at (1.5,0.5) {0};
    \node at (2.5,0.5) {0};
    \node at (3.5,0.5) {0};
    \node at (4.5,0.5) {0};

\end{scope}
\end{tikzpicture}
\vspace{1cm}

\item[(g)]
\begin{tikzpicture}
\begin{scope}[scale=0.5]
    \draw[fill, color=green!33] (0,2) rectangle (1,4);
    \draw[fill, color=green!33] (1,3) rectangle (3,4);
   
     \draw[fill, color=gray!33] (0,0) rectangle (2,1);
    \draw[fill, color=gray!33] (2,0) rectangle (3,3);
    \draw[fill, color=gray!33] (3,2) rectangle (5,3);
    \draw[fill, color=gray!33] (4,3) rectangle (5,4);
     \draw[fill, color=red!33] (3,0) rectangle (5,2);

    \draw[very thin, color=gray!100] (0,0) grid (5,4);
    \draw[very thick, violet] (0,0) -- (0,1) -- (2,1) -- (2,3) -- (3,3) -- (3,4) -- (5,4) ;

    \node at (0.5,3.5) {0};
    \node at (1.5,3.5) {0};
    \node at (2.5,3.5) {0};
    \node at (3.5,3.5) {1};
    \node at (4.5,3.5) {0};

    \node at (0.5,2.5) {0};
    \node at (1.5,2.5) {0};
    \node at (2.5,2.5) {1};
    \node at (3.5,2.5) {-1};
    \node at (4.5,2.5) {0};

    \node at (0.5,1.5) {0};
    \node at (1.5,1.5) {0};
    \node at (2.5,1.5) {0};
    \node at (3.5,1.5) {0};
    \node at (4.5,1.5) {0};

    \node at (0.5,0.5) {1};
    \node at (1.5,0.5) {0};
    \node at (2.5,0.5) {-1};
    \node at (3.5,0.5) {0};
    \node at (4.5,0.5) {0};

\end{scope}
\end{tikzpicture}
\vspace{1cm}

\item[(b)]
\begin{tikzpicture}
\begin{scope}[scale=0.5]
    \draw[fill, color=green!33] (0,2) rectangle (1,4);
    \draw[fill, color=green!33] (1,3) rectangle (3,4);
   
     \draw[fill, color=gray!33] (0,0) rectangle (2,1);
    \draw[fill, color=gray!33] (2,0) rectangle (3,3);
    \draw[fill, color=gray!33] (3,2) rectangle (5,3);
    \draw[fill, color=gray!33] (4,3) rectangle (5,4);
     \draw[fill, color=red!33] (3,0) rectangle (5,2);

    \draw[very thin, color=gray!100] (0,0) grid (5,4);
    \draw[very thick, violet] (0,0) -- (0,2) -- (1,2) -- (1,3) -- (4,3) -- (4,4) -- (5,4) ;

    \node at (0.5,3.5) {0};
    \node at (1.5,3.5) {0};
    \node at (2.5,3.5) {0};
    \node at (3.5,3.5) {0};
    \node at (4.5,3.5) {1};

    \node at (0.5,2.5) {0};
    \node at (1.5,2.5) {1};
    \node at (2.5,2.5) {0};
    \node at (3.5,2.5) {0};
    \node at (4.5,2.5) {-1};

    \node at (0.5,1.5) {1};
    \node at (1.5,1.5) {-1};
    \node at (2.5,1.5) {0};
    \node at (3.5,1.5) {0};
    \node at (4.5,1.5) {0};

    \node at (0.5,0.5) {0};
    \node at (1.5,0.5) {0};
    \node at (2.5,0.5) {0};
    \node at (3.5,0.5) {0};
    \node at (4.5,0.5) {0};

\end{scope}
\end{tikzpicture}
\vspace{1cm}

\item[(e)]
\begin{tikzpicture}
\begin{scope}[scale=0.5]
    \draw[fill, color=green!33] (0,2) rectangle (1,4);
    \draw[fill, color=green!33] (1,3) rectangle (3,4);
   
    \draw[fill, color=gray!33] (0,0) rectangle (2,1);
    \draw[fill, color=gray!33] (2,0) rectangle (3,3);
    \draw[fill, color=gray!33] (3,2) rectangle (5,3);
    \draw[fill, color=gray!33] (4,3) rectangle (5,4);
     \draw[fill, color=red!33] (3,0) rectangle (5,2);
  
    \draw[very thin, color=gray!100] (0,0) grid (5,4);
    \draw[very thick, violet] (0,0) -- (0,1) -- (1,1) -- (1,3) -- (3,3) -- (3,4) -- (5,4) ;

    \node at (0.5,3.5) {0};
    \node at (1.5,3.5) {0};
    \node at (2.5,3.5) {0};
    \node at (3.5,3.5) {1};
    \node at (4.5,3.5) {0};

    \node at (0.5,2.5) {0};
    \node at (1.5,2.5) {1};
    \node at (2.5,2.5) {0};
    \node at (3.5,2.5) {-1};
    \node at (4.5,2.5) {0};

    \node at (0.5,1.5) {0};
    \node at (1.5,1.5) {0};
    \node at (2.5,1.5) {0};
    \node at (3.5,1.5) {0};
    \node at (4.5,1.5) {0};

    \node at (0.5,0.5) {1};
    \node at (1.5,0.5) {-1};
    \node at (2.5,0.5) {0};
    \node at (3.5,0.5) {0};
    \node at (4.5,0.5) {0};

\end{scope}
\end{tikzpicture}
\vspace{1cm}

\item[(h)]
\begin{tikzpicture}
\begin{scope}[scale=0.5]
    \draw[fill, color=green!33] (0,2) rectangle (1,4);
    \draw[fill, color=green!33] (1,3) rectangle (3,4);
   
    \draw[fill, color=gray!33] (0,0) rectangle (2,1);
    \draw[fill, color=gray!33] (2,0) rectangle (3,3);
    \draw[fill, color=gray!33] (3,2) rectangle (5,3);
    \draw[fill, color=gray!33] (4,3) rectangle (5,4);
     \draw[fill, color=red!33] (3,0) rectangle (5,2);
 
    \draw[very thin, color=gray!100] (0,0) grid (5,4);
    \draw[very thick, violet] (0,0) -- (0,1) -- (1,1) -- (1,2) -- (2,2) -- (2,3) -- (3,3) -- (3,4) -- (5,4) ;

    \node at (0.5,3.5) {0};
    \node at (1.5,3.5) {0};
    \node at (2.5,3.5) {0};
    \node at (3.5,3.5) {1};
    \node at (4.5,3.5) {0};

    \node at (0.5,2.5) {0};
    \node at (1.5,2.5) {0};
    \node at (2.5,2.5) {1};
    \node at (3.5,2.5) {-1};
    \node at (4.5,2.5) {0};

    \node at (0.5,1.5) {0};
    \node at (1.5,1.5) {1};
    \node at (2.5,1.5) {-1};
    \node at (3.5,1.5) {0};
    \node at (4.5,1.5) {0};

    \node at (0.5,0.5) {1};
    \node at (1.5,0.5) {-1};
    \node at (2.5,0.5) {0};
    \node at (3.5,0.5) {0};
    \node at (4.5,0.5) {0};

\end{scope}
\end{tikzpicture}

\vspace{1cm}

\item[(c)]
\begin{tikzpicture}
\begin{scope}[scale=0.5]
    \draw[fill, color=green!33] (0,2) rectangle (1,4);
    \draw[fill, color=green!33] (1,3) rectangle (3,4);
   
   \draw[fill, color=gray!33] (0,0) rectangle (2,1);
    \draw[fill, color=gray!33] (2,0) rectangle (3,3);
    \draw[fill, color=gray!33] (3,2) rectangle (5,3);
    \draw[fill, color=gray!33] (4,3) rectangle (5,4);
     \draw[fill, color=red!33] (3,0) rectangle (5,2);
 
    \draw[very thin, color=gray!100] (0,0) grid (5,4);
    \draw[very thick, violet] (0,0) -- (0,2) -- (2,2) -- (2,3) -- (3,3) -- (3,4) -- (5,4) ;

    \node at (0.5,3.5) {0};
    \node at (1.5,3.5) {0};
    \node at (2.5,3.5) {0};
    \node at (3.5,3.5) {1};
    \node at (4.5,3.5) {0};

    \node at (0.5,2.5) {0};
    \node at (1.5,2.5) {0};
    \node at (2.5,2.5) {1};
    \node at (3.5,2.5) {-1};
    \node at (4.5,2.5) {0};

    \node at (0.5,1.5) {1};
    \node at (1.5,1.5) {0};
    \node at (2.5,1.5) {-1};
    \node at (3.5,1.5) {0};
    \node at (4.5,1.5) {0};

    \node at (0.5,0.5) {0};
    \node at (1.5,0.5) {0};
    \node at (2.5,0.5) {0};
    \node at (3.5,0.5) {0};
    \node at (4.5,0.5) {0};

\end{scope}
\end{tikzpicture}
\vspace{1cm}

\item[(f)]
\begin{tikzpicture}
\begin{scope}[scale=0.5]
    \draw[fill, color=green!33] (0,2) rectangle (1,4);
    \draw[fill, color=green!33] (1,3) rectangle (3,4);
   
     \draw[fill, color=gray!33] (0,0) rectangle (2,1);
    \draw[fill, color=gray!33] (2,0) rectangle (3,3);
    \draw[fill, color=gray!33] (3,2) rectangle (5,3);
    \draw[fill, color=gray!33] (4,3) rectangle (5,4);
     \draw[fill, color=red!33] (3,0) rectangle (5,2);
  
    \draw[very thin, color=gray!100] (0,0) grid (5,4);
    \draw[very thick, violet] (0,0) -- (0,1) -- (1,1) -- (1,3) -- (4,3) -- (4,4) -- (5,4) ;

    \node at (0.5,3.5) {0};
    \node at (1.5,3.5) {0};
    \node at (2.5,3.5) {0};
    \node at (3.5,3.5) {0};
    \node at (4.5,3.5) {1};

    \node at (0.5,2.5) {0};
    \node at (1.5,2.5) {1};
    \node at (2.5,2.5) {0};
    \node at (3.5,2.5) {0};
    \node at (4.5,2.5) {-1};

    \node at (0.5,1.5) {0};
    \node at (1.5,1.5) {0};
    \node at (2.5,1.5) {0};
    \node at (3.5,1.5) {0};
    \node at (4.5,1.5) {0};

    \node at (0.5,0.5) {1};
    \node at (1.5,0.5) {-1};
    \node at (2.5,0.5) {0};
    \node at (3.5,0.5) {0};
    \node at (4.5,0.5) {0};

\end{scope}
\end{tikzpicture}
\vspace{1cm}

\item[(i)]
\begin{tikzpicture}
\begin{scope}[scale=0.5]
    \draw[fill, color=green!33] (0,2) rectangle (1,4);
    \draw[fill, color=green!33] (1,3) rectangle (3,4);
   
   \draw[fill, color=gray!33] (0,0) rectangle (2,1);
    \draw[fill, color=gray!33] (2,0) rectangle (3,3);
    \draw[fill, color=gray!33] (3,2) rectangle (5,3);
    \draw[fill, color=gray!33] (4,3) rectangle (5,4);
     \draw[fill, color=red!33] (3,0) rectangle (5,2);

    \draw[very thin, color=gray!100] (0,0) grid (5,4);
    \draw[very thick, violet] (0,0) -- (0,1) -- (1,1) -- (1,2) -- (2,2) -- (2,3) -- (4,3) -- (4,4) -- (5,4) ;

    \node at (0.5,3.5) {0};
    \node at (1.5,3.5) {0};
    \node at (2.5,3.5) {0};
    \node at (3.5,3.5) {0};
    \node at (4.5,3.5) {1};

    \node at (0.5,2.5) {0};
    \node at (1.5,2.5) {0};
    \node at (2.5,2.5) {1};
    \node at (3.5,2.5) {0};
    \node at (4.5,2.5) {-1};

    \node at (0.5,1.5) {0};
    \node at (1.5,1.5) {1};
    \node at (2.5,1.5) {-1};
    \node at (3.5,1.5) {0};
    \node at (4.5,1.5) {0};

    \node at (0.5,0.5) {1};
    \node at (1.5,0.5) {-1};
    \node at (2.5,0.5) {0};
    \node at (3.5,0.5) {0};
    \node at (4.5,0.5) {0};

\end{scope}
\end{tikzpicture}
\end{multicols}
\vspace{.8cm}

\begin{multicols}{3}
\hspace{3.5in}
 \item[(j)]
\begin{tikzpicture}
\begin{scope}[scale=0.5]
    \draw[fill, color=green!33] (0,2) rectangle (1,4);
    \draw[fill, color=green!33] (1,3) rectangle (3,4);
   
   \draw[fill, color=gray!33] (0,0) rectangle (2,1);
    \draw[fill, color=gray!33] (2,0) rectangle (3,3);
    \draw[fill, color=gray!33] (3,2) rectangle (5,3);
    \draw[fill, color=gray!33] (4,3) rectangle (5,4);
     \draw[fill, color=red!33] (3,0) rectangle (5,2);
 
    \draw[very thin, color=gray!100] (0,0) grid (5,4);
    \draw[very thick, violet] (0,0) -- (0,1) -- (2,1) -- (2,3) -- (4,3) -- (4,4) -- (5,4) ;

    \node at (0.5,3.5) {0};
    \node at (1.5,3.5) {0};
    \node at (2.5,3.5) {0};
    \node at (3.5,3.5) {0};
    \node at (4.5,3.5) {1};

    \node at (0.5,2.5) {0};
    \node at (1.5,2.5) {0};
    \node at (2.5,2.5) {1};
    \node at (3.5,2.5) {0};
    \node at (4.5,2.5) {-1};

    \node at (0.5,1.5) {0};
    \node at (1.5,1.5) {0};
    \node at (2.5,1.5) {0};
    \node at (3.5,1.5) {0};
    \node at (4.5,1.5) {0};

    \node at (0.5,0.5) {1};
    \node at (1.5,0.5) {0};
    \node at (2.5,0.5) {-1};
    \node at (3.5,0.5) {0};
    \node at (4.5,0.5) {0};

\end{scope}

\end{tikzpicture}
\hspace{3.5in}

\end{multicols}
\end{enumerate}
\caption{All matrices in $\pasm_{4,5}^{(4,2,2)  /(3,1)}$ with  $\mu$ between $\lambda=(3,1)$ and $\nu=(4,2,2)$, where $\mu$ is in violet. Note: $\lambda$ is green, the border strip associated to $\nu$ is in gray.}
\label{fig:4x5examples}
\end{figure}

\begin{theorem}
\label{thm:ineq_desc}
Let $\lambda,\nu$ be partitions inside the partition $(n-1)^{m-1}$ such that $\lambda\subseteq\nu$. Then $\pasm(\nu/\lambda,m,n)$ is given by the inequalities: 
\begin{itemize}

\item Partial column and partial row sums are between $0$ and $1$:
\begin{equation}
\label{pasmeq2}
0 \leq \displaystyle\sum_{i'=1}^{i} X_{i'j} \leq 1, \text{ for all } 1 \leq i \leq m, 1 \leq j \leq n.
\end{equation}

\begin{equation}
\label{pasmeq3}
0 \leq \displaystyle\sum_{j'=1}^{j} X_{ij'} \leq 1, \text{ for all }  1 \leq i \leq m, 1 \leq j \leq n.
\end{equation}
\item First column and first row sums are $1$: 
\begin{equation}
\label{pasmeq4}
\displaystyle\sum_{i=1}^{m} X_{i1} = 1 \text{ and } \displaystyle\sum_{j=1}^{n} X_{1j} = 1
\end{equation}


\item Other columns and rows sum to $0$:
\begin{equation}
\label{pasmeq6}
\displaystyle\sum_{j=1}^{n} X_{ij} = 0 \text{ for } 2 \leq i \leq m \text{ and } \sum_{i=1}^{m} X_{ij} = 0 \text{ for } 2 \leq j \leq n
\end{equation}

\item The entries of $\lambda$ are fixed to be $0$:
\begin{equation}
\label{pasmeq8}
X_{ij} = 0 \text{ if } j\leq \lambda_{i} \text{ for } 1\leq i\leq n
\end{equation}

\item The entries southeast of the border strip associated to $\nu$ are fixed to be $0$: 
\begin{equation}
\label{pasmeq7}
X_{ij} = 0 \text{ if } \nu_{i-1} < j \leq n \text{ for } i>1, \text{ and } X_{ij} = 0 \text{ if } \nu_1 +1 < j  \leq n
\end{equation}

\end{itemize}
\end{theorem}

\begin{proof}
Fix $m,n>0$ and partitions $\lambda\subseteq\nu\subset (n-1)^{m-1}$. 

Let $\P(\nu / \lambda,m,n)$ be the polytope with inequality description given by $(\ref{pasmeq2}) - (\ref{pasmeq7})$. We wish to show $\pasm(\nu / \lambda,m,n)$ is contained in $\P(\nu / \lambda,m,n)$. Remark~\ref{remark:pasmmn} notes that $\pasm(\nu / \lambda,m,n)$ is contained in $\pasm(m,n)$. By \cite[Theorem 4.6]{HeuerStriker}, the polytope $\pasm(m,n)$ satisfies (\ref{pasmeq2}) and (\ref{pasmeq3}). Let $X \in \pasm(\nu / \lambda, m,n)$. Then $X$ is a convex combination $X=\displaystyle\sum_{\lambda \subseteq \mu \subseteq \nu}c_{\mu}M^\mu$ of matrices $M^\mu$ with $c_{\mu}>0$ and $\displaystyle\sum_{\lambda \subseteq \mu \subseteq \nu} c_{\mu}=1$. 

Each $M^\mu$ has first row sum $1$, first column sum $1$, and all other row or column sums $0$ by Lemma~\ref{lem:Mmu}.
%
%
Then \[\sum_{j=1}^m X_{1j} = \sum_{j=1}^m \sum_{\lambda \subseteq \mu \subseteq \nu}c_{\mu}M^\mu_{1j}= \sum_{\lambda \subseteq \mu \subseteq \nu}c_{\mu}\sum_{j=1}^m M^\mu_{1j}=\sum_{\lambda \subseteq \mu \subseteq \nu}c_{\mu}\cdot 1=1\]
and
\[\sum_{i=1}^n X_{i1} = \sum_{i=1}^n \sum_{\lambda \subseteq \mu \subseteq \nu}c_{\mu}M^\mu_{i1}= \sum_{\lambda \subseteq \mu \subseteq \nu}c_{\mu}\sum_{i=1}^n M^\mu_{i1}=\sum_{\lambda \subseteq \mu \subseteq \nu}c_{\mu}\cdot 1=1.\]
So (\ref{pasmeq4}) is satisfied.

Similarly, if $1<k\leq m$,
\[\sum_{j=1}^m X_{kj} = \sum_{j=1}^m \sum_{\lambda \subseteq \mu \subseteq \nu}c_{\mu}M^\mu_{kj}= \sum_{\lambda \subseteq \mu \subseteq \nu}c_{\mu}\sum_{j=1}^m M^\mu_{kj}=\sum_{\lambda \subseteq \mu \subseteq \nu}c_{\mu}\cdot 0=0\]
and if $1<k\leq n$,
\[\sum_{i=1}^n X_{ik} = \sum_{i=1}^n \sum_{\lambda \subseteq \mu \subseteq \nu}c_{\mu}M^\mu_{ik}= \sum_{\lambda \subseteq \mu \subseteq \nu}c_{\mu}\sum_{i=1}^n M^\mu_{ik}=\sum_{\lambda \subseteq \mu \subseteq \nu}c_{\mu}\cdot 0=0.\]
So (\ref{pasmeq6}) is satisfied.

Since $\lambda \subseteq \mu$, and the $1$'s and $-1$'s in $M^{\mu}$ have been placed outside of $\mu$, all entries in $\lambda$ are equal to $0$. So (\ref{pasmeq8}) is satisfied. Also, since $\mu \subseteq \nu$ and all $1$'s and $-1$'s in $M^{\mu}$ have been placed at most one box south and/or east of $\mu$, all entries southeast of the border strip associated to $\nu$ are equal to $0$. So (\ref{pasmeq7}) is satisfied. 

Now we wish to show that $\P(\nu / \lambda, m, n)$ is contained in $\pasm(\nu / \lambda,m,n)$. Suppose that $X \in \P(\nu / \lambda, m, n)$. By (3.1) and (3.2), $X$ is also in $\pasm(m,n)$. So we can write it as the convex combination of some partial ASMs: $X=\displaystyle\sum_{k=1}^w c_{k}M^k$ with $c_{k}>0$ and $\displaystyle\sum_{k=1}^w c_{k}=1$.

Suppose one of these partial alternating sign matrices does not satisfy the first equality in (\ref{pasmeq4}). Without loss of generality, say it is $M^1$. Since it is a partial alternating sign matrix that does not have a first column total sum of 1, then it must have a first column sum of 0. 
So \[\displaystyle\sum_{i=1}^{m} X_{i1} = \displaystyle\sum_{i =1}^{m} \displaystyle\sum_{k=1}^w c_{k}M^k_{i1} =   \displaystyle\sum_{k=1}^w c_{k}\displaystyle\sum_{i=1}^{m} M^k_{i1} = c_1 \displaystyle\sum_{i=1}^{m} M^1_{i1} + \displaystyle\sum_{k=2}^{w} c_k \displaystyle\sum_{i=1}^{m} M^k_{i1} = c_1 \cdot 0 + \displaystyle\sum_{k=2}^{w} c_k\cdot 1 < 1.\] This is a contradiction. We can do a similar proof for the second equality in (\ref{pasmeq4}). So each $M^k$ satisfies (\ref{pasmeq4}). 

Now suppose one of these partial alternating sign matrices does not satisfy the first equality in (\ref{pasmeq6})  for a fixed value of $i$ with $2\leq i\leq m$. Without loss of generality, say this matrix is $M^1$. Since this is a partial alternating sign matrix that does not have an $i$th row total sum of 0, then it must have an $i$th row total sum of 1. 
So \[\displaystyle\sum_{j=1}^n X_{ij} = \displaystyle\sum_{j=1}^n  \displaystyle\sum_{k=1}^w c_{k}M^k_{ij} =   \displaystyle\sum_{k=1}^w c_{k}\displaystyle\sum_{j=1}^n  M^k_{ij} = c_1 \displaystyle\sum_{j=1}^n  M^1_{ij} + \displaystyle\sum_{k=2}^{w} c_k \displaystyle\sum_{j=1}^n  M^k_{ij} = c_1 \cdot 1 + \displaystyle\sum_{k=2}^{w} c_k\cdot 0 = c_1 > 0.\] This is a contradiction. We can do a similar proof for the second equality in (\ref{pasmeq6}). So each $M^k$ satisfies (\ref{pasmeq6}).

Now suppose one of these partial alternating sign matrices, call it $M^1$, does not satisfy the  (\ref{pasmeq8})  for some $j<\lambda_i$. Suppose this is the smallest such value of $j$ with this $i$. So this is a partial ASM with $i$th row partial sum $\displaystyle\sum_{j'=1}^j M^1_{ij} = 1$ but $\displaystyle\sum_{j'=1}^{j-1} M^k_{ij} = 0$ for all $k$. But since $X$ satisfies (\ref{pasmeq8}), $\displaystyle\sum_{j'=1}^j X_{ij} = 0$. 
So \[\displaystyle\sum_{j'=1}^j X_{ij} = \displaystyle\sum_{j'=1}^j  \displaystyle\sum_{k=1}^w c_{k}M^k_{ij} =   \displaystyle\sum_{k=1}^w c_{k}\displaystyle\sum_{j'=1}^j  M^k_{ij} =  c_1 \displaystyle\sum_{j'=1}^j  M^1_{ij} + \displaystyle\sum_{k=2}^{w} c_k \displaystyle\sum_{j'=1}^j  M^k_{ij} =  c_1 \cdot 1 + \displaystyle\sum_{k=2}^{w} c_k\cdot 0 = c_1 > 0.\] This is a contradiction. So each $M^k$ satisfies (\ref{pasmeq8}). We can do a similar proof for (\ref{pasmeq7}) by choosing the largest such $j$ rather than the smallest.

Thus we have shown both $\P(\nu / \lambda, m, n)\subseteq \pasm(\nu / \lambda,m,n)$ and $\pasm(\nu / \lambda,m,n)\subseteq \P(\nu / \lambda, m, n)$, verifying these polytopes are equal.
\end{proof}

\section{The $(\nu / \lambda)$-partial ASM polytope is a face of the partial ASM polytope}
\label{sec:face}
In this section, we use the characterization of the face lattice of $\pasm(m,n)$ shown in \cite{HeuerStriker} to show that $\pasm(\nu / \lambda,m,n)$ is a face of $\pasm(m,n)$ in Theorem~\ref{thm:face}. In order to prove this result, we use the following definitions slightly adapted from \cite{HeuerStriker}, which also modify analogous constructs from \cite{SolhjemStriker}.

\begin{definition}
Define the $m \times n$ \emph{grid graph} $\Gamma_{(m,n)}$ as follows.
The vertex set is given as $V(m,n):=\{(i,j)$ : $1 \leq i \leq m+1$, $1 \leq j \leq n+1  \}$.
The edge set is:
\[E(m,n):=  \begin{cases} (i,j) \text{ to } (i+1,j) & 1 \leq i \leq m, 1 \leq j \leq n\\ (i,j) \text{ to } (i,j+1) & 1 \leq i \leq m, 1 \leq j \leq n. \end{cases}\]
The graph is drawn with $j$ increasing to the right and $i$ increasing down to correspond with matrix indexing.

Given an $m \times n$ matrix $X$, we define a labeled graph, $g(X)$, which is a labeling of the edges of $\Gamma_{(m,n)}$. 
The horizontal edges from $(i,j)$ to $(i,j+1)$ are each labeled by the corresponding row partial sum $\displaystyle\sum_{j'=1}^{j} X_{ij'}$ ($1\leq i \leq m, 1 \leq j \leq n$). Similarly, the vertical edges from $(i,j)$ to $(i+1,j)$ are each labeled by the corresponding column partial sum $\displaystyle\sum_{i'=1}^{i} X_{i'j}$ ($1\leq i \leq m, 1 \leq j \leq n$).
\end{definition}
\begin{definition}
Given $M\in\pasm_{m,n}$, we call $g(M)$ the \emph{basic sum-labeling} of $M$. 
Given a collection of partial alternating sign matrices $\mathcal{M}=\{M_1,M_2,\dots,M_r\}\subseteq\pasm_{m,n}$, 
define the \emph{sum-labeling} $g(\mathcal{M})=\displaystyle\bigcup_{i=1}^r g(M_i)$, where the union is defined as the labeling of $\Gamma_{(m,n)}$ whose edge labels are given by the union of the corresponding edge labels of the $g(M_i)$. The set of sum-labelings of $\Gamma_{(m,n)}$ is the set of $g(\mathcal{M})$ over all collections $\mathcal{M}=\{M_1,M_2,\dots,M_r\}\subseteq\pasm_{m,n}$, $(0\leq r)$. 
\end{definition}

\begin{definition}
Given a sum-labeling $\delta$, consider the planar graph $G$ composed of the edges of $\delta$ labeled by the two-element set $\{0,1\}$ (and all incident vertices).
A \emph{region} of $\delta$ is defined as a planar region of $G$, excluding the exterior region. 
Let $\mathcal{R}(\delta)$ denote the number of regions of $\delta$. 
\end{definition}

These definitions relate to the following result on the face lattice of $\pasm(m,n)$.

\begin{figure}
\begin{center}
\includegraphics[scale=.8]{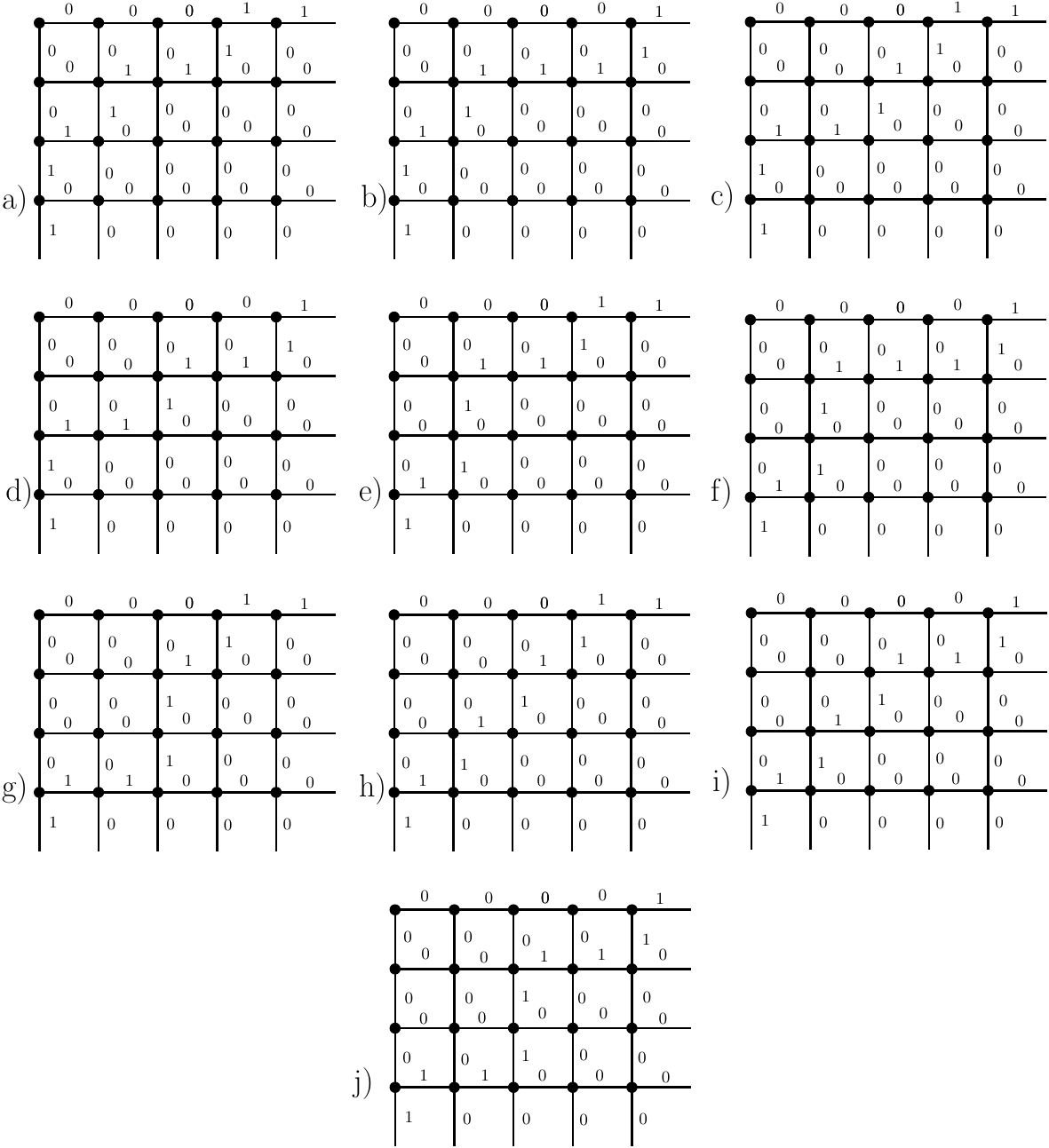}
\end{center}
\caption{Examples of basic sum-labelings, corresponding to the matrices in Figure \ref{fig:4x5examples}, where $\mu$ is between $\lambda=(3,1)$ and $\nu=(4,2,2)$.}
\label{fig:sum_labelings}
\end{figure}

\begin{theorem}[Theorem 4.18 of \cite{HeuerStriker}]
\label{thm:pasm_facelattice}
Let $F$ be a face of $\pasm(m,n)$ and $\mathcal{M}(F)$ be the set of partial alternating sign matrices that are vertices of $F$. The map $\psi:F\mapsto g(\mathcal{M}(F))$ induces an isomorphism between the face lattice of $\pasm(m,n)$ and the set of sum-labelings of ${\Gamma}_{(m,n)}$ ordered by containment. Moreover, 
$\mathrm{dim}(F)= \mathcal{R}(\psi(F))$.
\end{theorem}

Using the characterization of faces of $\pasm(m,n)$ as sum-labelings in the above theorem, we show that $\pasm(\nu / \lambda,m,n)$ is a face of $\pasm(m,n)$.

\begin{theorem}
\label{thm:face}
Given $\lambda \subseteq \nu \subseteq (n-1)^{m-1}$. Then $\pasm(\nu / \lambda,m,n)$ is a face of $\pasm(m,n)$ of dimension $|\nu|-|\lambda|$.
\end{theorem}
\begin{proof}
By Theorem \ref{thm:pasm_facelattice}, the faces of $\pasm(m,n)$ are in bijection with sum-labelings of $\Gamma_{(m,n)}$. Since, by Definition~\ref{def:poly_nu_lambda}, $\pasm(\nu / \lambda,m,n)$ is the convex hull of the partial alternating sign matrices $\pasm^{\nu / \lambda}_{m,n}$, we need only show there exists a sum-labeling whose contained basic sum-labelings correspond exactly to these partial alternating sign matrices. We give this sum-labeling explicitly.




Given a partition $\mu\subset (n-1)^{m-1}$, define $E_\mu$ to be the edge set which ``outlines'' $\mu$ in $\Gamma_{m,n}$.
Horizontal edges in $E_\mu$ are given as: $(i, j)$ to $(i, j+1)$ for $1 \leq i \leq m$, $\mu_i + 1 \leq j \leq \mu_{i-1}$, where we take $\mu_0=n$.
Vertical edges in $E_\mu$ are given as: $(i, \mu_i +1)$ to $(i+1, \mu_i +1)$ for $1 \leq i \leq m$.

To create the desired sum-labeling, label the edges in $\displaystyle\left(\bigcup_{\lambda \subseteq \mu \subseteq \nu} E_\mu \right) - (E_\lambda \cap E_\nu)$  with  $\{0,1\}$ (the blue lines in Figure~\ref{fig:4x5unionsums}). Label the edges in the intersection $E_\lambda \cap E_\nu$ with $1$. Label all remaining edges of $\Gamma_{(m,n)}$ with $0$. See Example~\ref{Emu_example} for an example of an explicit construction of $E_\lambda$ and $E_\nu$.



We justify that this is the correct sum-labeling as follows. 
For every $M \in \pasm_{\nu / \lambda}$, the entries in $M$ corresponding to $\lambda$ and those outside the border strip associated to $\nu$ are fixed to be zero. Since the first row must sum to one, we have that the partial row sum for the first row will always be $0$ for the first $\lambda_1$ positions, $1$ for the last $n-\nu_1$ positions, and could be either $0$ or $1$ for those positions in between. 
Since all rows except the first must sum to $0$, we have that these partial row sums must be $0$ for the first $\lambda_i$ entries of row $i$ $(i>1)$ as well as the last $n-\nu_{i-1}$ entries, and may be $0$ or $1$ in between. 
A similar argument can be made for the partial columns sums. The union of all of the corresponding basic sum-labelings for these matrices results in the sum-labeling described above. See Figures~\ref{fig:sum_labelings} and \ref{fig:4x5unionsums} for an example.

By construction, this sum-labeling has $|\nu|-|\lambda|$ regions (exactly the boxes which correspond to $\nu / \lambda$ in $\Gamma_{(m,n)}$). Thus by Theorem \ref{thm:pasm_facelattice}, the dimension of $\pasm(\nu / \lambda,m,n)$ equals $|\nu|-|\lambda|$.
\end{proof}

\begin{example}
\label{Emu_example}
In Figure~\ref{fig:4x5unionsums}, we have constructed the basic sum-labeling corresponding to $\pasm_{\nu / \lambda}$ where $\lambda=(3,1)$ and $\nu = (4,2,2)$.
We show this construction as follows; see Figure~\ref{fig:4x5elamenu}.

Horizontal edges in $E_\lambda$ are $(1,4)$ to $(1,5)$, $(1,5)$ to $(1,6)$, $(2,2)$ to $(2,3)$, $(2,3)$ to $(2,4)$, $(3,1)$ to $(3,2)$. Note that there is no horizontal edge in row 4 because there are no $j$ such that $\lambda_4+1 \leq j \leq \lambda_3$ (since this gives $1 \leq j \leq 0$).
Vertical edges in $E_\lambda$ are $(1,4)$ to $(2,4)$, $(2,2)$ to $(3,2)$, $(3,1)$ to $(4,1)$, and $(4,1)$ to $(5,1)$. Note that we are taking $\lambda_3 = \lambda_4 = 0$.

Horizontal edges in $E_\nu$ are $(1,5)$ to $(1,6)$, $(2,3)$ to $(2,4)$, $(2,4)$ to $(2,5)$, $(4,1)$ to $(4,2)$, and $(4,2)$ to $(4,3)$. Note that there is no horizontal edge in row 3 because there are no $j$ such that $\nu_3+1 \leq j \leq \nu_2$ (since this gives $3 \leq j \leq 2$).
Vertical edges in $E_\nu$ are $(1,5)$ to $(2,5)$, $(2,3)$ to $(3,3)$, $(3,3)$ to $(4,3)$, and $(4,1)$ to $(5,1)$. Again, note that we take $\nu_4 = 0$.

So then $E_\lambda \cap E_\nu$ consists of the edges $(4,1)$ to $(5,1)$, $(2,3)$ to $(2,4)$, and $(1,5)$ to $(1,6)$. These will have label $1$. The remaining edges in $E_\lambda$ and $E_\nu$ have label $\{0,1\}$ and form the ``outline'' of the boxes between them (which gets ``filled in'' with the $\mu$ between $\lambda$ and $\nu$).
Since $|\lambda|=4$ and $|\nu|=8$, we see that the dimension of the face $\pasm(\nu / \lambda,4,5)$ of $\pasm(4,5)$ is of dimension $|\nu|-|\lambda|=8-4=4$, the number of regions in Figure~\ref{fig:4x5unionsums}.
\end{example}

\begin{figure}
\includegraphics[scale=.8]{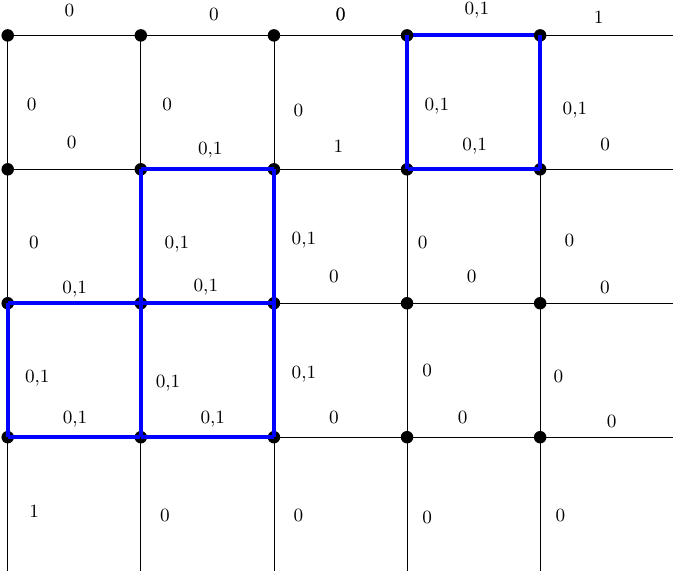}
\caption{The union of the sum-labelings in Figure \ref{fig:4x5examples}, with  $\mu$ between $\lambda=(3,1)$ and $\nu=(4,2,2)$.}
\label{fig:4x5unionsums}
\end{figure}

\begin{figure}[htbp]
\includegraphics[scale=.65]{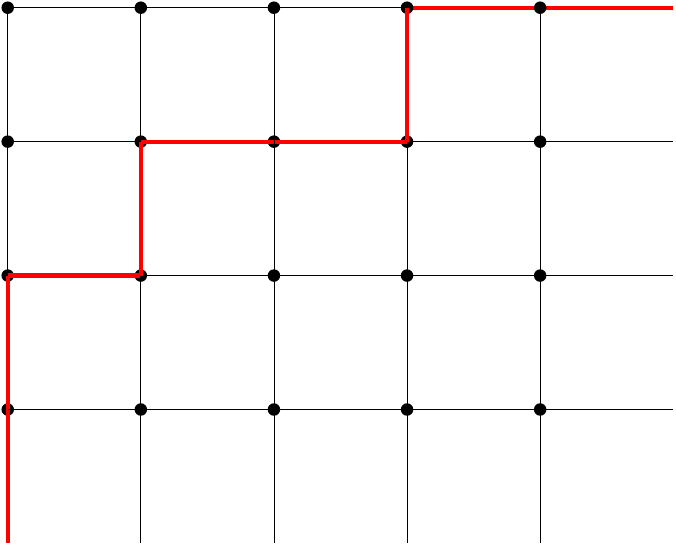} \hspace{.2in}\includegraphics[scale=.65]{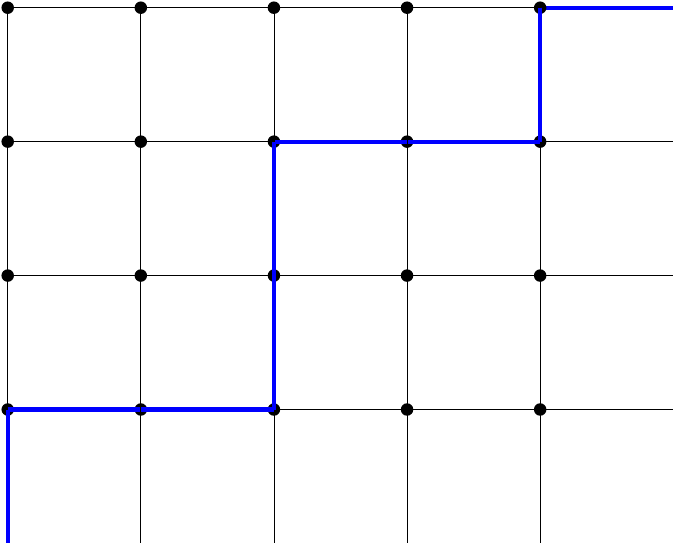}

\caption{Left: $E_{(3,1)}$ is highlighted in red. Right: $E_{(4,2,2)}$ is highlighted in blue.}
\label{fig:4x5elamenu}
\end{figure}

\section{The $(\nu / \lambda)$-partial ASM polytope is an order polytope and flow polytope}
\label{sec:orderflow}
In this section, we show $\pasm(\nu / \lambda,m,n)$ is an order polytope of a strongly planar graph, and so is consequently a flow polytope. We also show how this polytope relates to the ASM-CRY polytope.

We will need the following definitions for our main result, Theorem~\ref{thm:int_equiv}.

\begin{definition}[\cite{Stop}]
The \emph{order polytope} $\mathcal{O}(P)$ of a poset $P$ with
elements $\{p_1,p_2,\ldots,p_d\}$  is the set of points
$(z_1,z_2,\ldots,z_d)$ in $\mathbb{R}^d$ with $0\leq z_i\leq 1$ and if
$p_i \leq p_j$ in $P$ then $z_i \leq z_j$. We identify each point
$(z_1,z_2,\ldots,z_d)$ of  $\mathcal{O}(P)$ with the function $f:P\to
\mathbb{R}$ where $f(p_i)=z_i$. 
\end{definition}

\begin{definition}
Given a matrix $(M_{ij})_{1\leq i \leq m, 1 \leq j \leq n}$, define the \emph{(northwest) corner sum matrix} $(c_{ij})_{1\leq i \leq m, 1 \leq j \leq n}$ by \[ c_{ij} = \sum_{1 \leq i' \leq i, 1 \leq j' \leq j} M_{i',j'}.\]
\end{definition}

\begin{definition}
Given a  partition $\nu$, let $P(\nu)$ denote the poset of shape $\nu$, where each box $(i,j)$ of $\nu$ is a poset element, the upper left box $(1,1)$ of $\nu$ is the minimal element of $P(\nu)$, and the partial order is given as $(i,j)\leq (i',j')$ if and only if $i\leq i'$ and $j\leq j'$. Define $P(\nu / \lambda)$ analogously by deleting the elements of $P(\nu)$ corresponding to $\lambda$.
\end{definition}

We use the corner sum map to construct a map from the polytope $\pasm(\nu / \lambda,m,n)$ to the order polytope  $\mathcal{O}(P(\nu / \lambda))$, which we show in Theorem~\ref{thm:int_equiv} is an integral equivalence. 
For $S\subseteq\mathbb{R}$, let $\mathcal{A}(P(\nu / \lambda), S)$ be the set of functions $g:P(\nu / \lambda)\rightarrow S$. In this notation, the order polytope $\mathcal{O}(P(\nu / \lambda))$ is a subset of $\mathcal{A}(P(\nu / \lambda), [0,1])$. Define $\Psi: \pasm(\nu / \lambda,m,n) \rightarrow \mathcal{A}(P(\nu / \lambda,m,n), \mathbb{R})$ by $a\rightarrow g_a$ where
$g_a(i, j) = c_{ij}$. See Figure \ref{fig:posetexample} for an example.

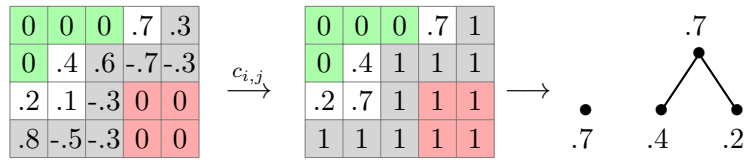
\begin{figure}[htbp]
\begin{tikzpicture}
\begin{scope}[scale=0.5]
    \draw[fill, color=green!33] (0,2) rectangle (1,4);
    \draw[fill, color=green!33] (1,3) rectangle (3,4);
   
    \draw[fill, color=gray!33] (0,0) rectangle (2,1);
    \draw[fill, color=gray!33] (2,0) rectangle (3,3);
    \draw[fill, color=gray!33] (3,2) rectangle (5,3);
    \draw[fill, color=gray!33] (4,3) rectangle (5,4);
    
    \draw[fill, color=red!33] (3,0) rectangle (5,2);
    \draw[very thin, color=gray!100] (0,0) grid (5,4);

    \node at (0.5,3.5) {0};
    \node at (1.5,3.5) {0};
    \node at (2.5,3.5) {0};
    \node at (3.5,3.5) {.7};
    \node at (4.5,3.5) {.3};

    \node at (0.5,2.5) {0};
    \node at (1.5,2.5) {.4};
    \node at (2.5,2.5) {.6};
    \node at (3.5,2.5) {-.7};
    \node at (4.5,2.5) {-.3};

    \node at (0.5,1.5) {.2};
    \node at (1.5,1.5) {.1};
    \node at (2.5,1.5) {-.3};
    \node at (3.5,1.5) {0};
    \node at (4.5,1.5) {0};

    \node at (0.5,0.5) {.8};
    \node at (1.5,0.5) {-.5};
    \node at (2.5,0.5) {-.3};
    \node at (3.5,0.5) {0};
    \node at (4.5,0.5) {0};

\end{scope}
\end{tikzpicture}
\raisebox{.8cm}{
$\overset{c_{i,j}}{\longrightarrow}$} \hspace{.06in}
\begin{tikzpicture}
\begin{scope}[scale=0.5]
    \draw[fill, color=green!33] (0,2) rectangle (1,4);
    \draw[fill, color=green!33] (1,3) rectangle (3,4);
   
    \draw[fill, color=gray!33] (0,0) rectangle (2,1);
    \draw[fill, color=gray!33] (2,0) rectangle (3,3);
    \draw[fill, color=gray!33] (3,2) rectangle (5,3);
    \draw[fill, color=gray!33] (4,3) rectangle (5,4);
    
    \draw[fill, color=red!33] (3,0) rectangle (5,2);
    \draw[very thin, color=gray!100] (0,0) grid (5,4);
      
    \node at (0.5,3.5) {0};
    \node at (1.5,3.5) {0};
    \node at (2.5,3.5) {0};
    \node at (3.5,3.5) {.7};
    \node at (4.5,3.5) {1};
   
    \node at (0.5,2.5) {0};
    \node at (1.5,2.5) {.4};
    \node at (2.5,2.5) {1};
    \node at (3.5,2.5) {1};
    \node at (4.5,2.5) {1};
    
    \node at (0.5,1.5) {.2};
    \node at (1.5,1.5) {.7};
    \node at (2.5,1.5) {1};
    \node at (3.5,1.5) {1};
    \node at (4.5,1.5) {1};
    
    \node at (0.5,0.5) {1};
    \node at (1.5,0.5) {1};
    \node at (2.5,0.5) {1};
    \node at (3.5,0.5) {1};
    \node at (4.5,0.5) {1};
    
\end{scope}
\end{tikzpicture}
\raisebox{.8cm}{$\longrightarrow$}
\begin{tikzpicture}
\begin{scope}[scale=0.5]
    \node at (0,0) {\textbullet};
    \node at (2,0) {\textbullet};
    \node at (4,0) {\textbullet};
    \node at (3,1.5) {\textbullet};
    
    \draw[thick, black] (2,0) -- (3,1.5) -- (4,0);
    
    \node at (-0.1,-0.75) {.7};
    \node at (1.9,-0.75) {.4};
    \node at (3.9,-0.75) {.2};
    \node at (2.9,2.25) {.7};
\end{scope}
\end{tikzpicture}

\caption{The map $\Psi$ from a point in the polytope $\pasm(\nu / \lambda,4,5)$ to a point in the order polytope  $\mathcal{O}(P(\nu / \lambda))$ for $\lambda = (3,1)$ and $\nu = (4,2,2)$.
The first step replaces each entry by the northwest corner sum. The second step labels the poset $P(\nu / \lambda)$ with the corresponding values.}
\label{fig:posetexample}
\end{figure}

Recall integer polytopes 
$\mathcal{P} \subset \mathbb{R}^d$ and $\mathcal{Q} \subset \mathbb{R}^r$ are integrally equivalent if there is an affine transformation
$f : \mathbb{R}^d \rightarrow \mathbb{R}^r$ such that $f$ maps $P$ bijectively onto $Q$ and $f$ maps $\mathbb{Z}^d \cap \aff(P)$ bijectively
onto $\mathbb{Z}^r \cap \aff(Q)$, where $\aff$ denotes affine span. Note that integrally equivalent polytopes have not only  isomorphic face lattices (combinatorial equivalence), but also the same volume and Ehrhart polynomial. 

Our main theorem is below.
\begin{theorem}
\label{thm:int_equiv}
The $(\nu / \lambda)$-partial ASM polytope $\pasm(\nu / \lambda,m,n)$ is integrally equivalent to the order polytope $\mathcal{O}(P(\nu / \lambda))$.
\end{theorem}

The following lemma shows the map $\Psi$ sends $\pasm(\nu / \lambda,m,n)$ into the order polytope. Then in the proof of Theorem~\ref{thm:int_equiv}, we show $\Psi$ is an integral equivalence.

\begin{lemma}\label{lem:psi}
The image of $\Psi$ is in the order polytope $\mathcal{O}(P(\nu / \lambda))$. 
\end{lemma}

\begin{proof}
We begin by showing the image of $\Psi$ is in $\mathcal{A}(P(\nu / \lambda), [0,1])$. We first note that $c_{ij}\geq 0$ by (\ref{pasmeq2}), since $c_{ij}$ is a sum of partial column sums, which are all nonnegative.

We now show $c_{ij}\leq 1$. The first row sums to $1$ by (\ref{pasmeq4}) and the other rows sum to $0$ by (\ref{pasmeq6}). Thus the total matrix sum is $1$, so $c_{mn}=1$.

Thus, the proof will be complete by showing $c_{ij}\leq c_{i'j'}$ whenever $i\leq i'$ and $j\leq j'$. 
Consider $c_{ij}$ and $c_{i+1,j}$. $c_{i+1,j}=c_{ij} + \displaystyle\sum_{j'=1}^j X_{i+1,j'}$. By (\ref{pasmeq2}), the partial column sum $\displaystyle\sum_{j'=1}^j X_{i+1,j'}\geq 0$, so $c_{ij}\leq c_{i+1,j}$. The inequality $c_{ij}\leq c_{i,j+1}$ follows analogously. Therefore, $\Psi$ maps $\pasm(\nu / \lambda,m,n)$ into the order polytope $\mathcal{O}(P(\nu / \lambda))$. 
\end{proof}

\begin{proof}[Proof of Theorem \ref{thm:int_equiv}]
By definition, the map $\Psi$ is an affine transformation of the form $X\mapsto AX$ where $A$ is a $0, 1$-matrix that is lower triangular with ones on the diagonal. So it is a unimodular transformation. By Lemma~\ref{lem:psi}, the image of this map is in the order polytope  $\mathcal{O}(P(\nu / \lambda))$.
Thus, $\Psi$ is a bijection between $\pasm(\nu / \lambda,m,n)$ and $\mathcal{O}(P(\nu / \lambda))$ that preserves their respective lattices. This
shows that the two polytopes are integrally equivalent. 
\end{proof}

By Stanley’s theory of order polytopes \cite[Theorem 3.9]{Stop} we express the volume and Ehrhart
polynomial of the polytopes in this family in terms of their associated posets. We use the following standard definitions (see e.g.~\cite[Section 4]{Stop}).
\begin{definition} 
Let $P$ be a finite poset with elements $\{p_1,\ldots,p_d\}$ and $t$ a positive integer. Define $\Omega(P, t)$ to be the number of order-preserving maps $\eta: P \rightarrow \{1, \ldots, t\}$. That is, if $p_i \leq p_j$ in $P$ then $\eta(p_i) \leq \eta(p_j)$. Then $\Omega(P, t)$ is a polynomial function of $t$ of degree $d$, called the \emph{order polynomial} of $P$. 

A \emph{linear extension} of a poset $P$ is an order-preserving bijection $\eta: P \rightarrow \{1, \ldots, d\}$. Let $e(P)$ denote the number of linear extensions of $P$.
\end{definition}


\begin{corollary}
\label{cor:volume}
The normalized volume of $\pasm(\nu / \lambda,m,n)$ is  $e(P(\nu / \lambda))$  and its Ehrhart polynomial is $L_{\pasm(\nu / \lambda,m,n)}(t) = \Omega(P(\nu / \lambda),t + 1)$. 
In particular, $e(P(\nu / \lambda))$ is the number of skew standard Young tableaux, enumerated by the Naruse hook-length formula
\[e(P(\nu / \lambda))=|\nu / \lambda|!\sum_{D\in\mathcal{E}(\nu / \lambda)}\prod_{u\in[\nu]\setminus D}\frac{1}{h(u)},\]
where $\mathcal{E}(\nu / \lambda)$ is the set of excited diagrams of $\nu / \lambda$ and $h(u)$ is the hook number of a box $u$.
\end{corollary}
\begin{proof}
The Naruse hook-length formula, found by Naruse~\cite{NaruseHLF} and proved in~\cite[Theorem 1.2]{MPP2018}, counts the number of skew standard Young tableaux of a given skew shape $\nu / \lambda$. These are  equivalent to linear extensions of $P(\nu / \lambda)$ by rotation. The result follows by Stanley’s theory of order polytopes \cite[Theorem 3.9]{Stop}.
\end{proof}

We now state a few more definitions from \cite{ASMCRY} that are needed for Corollary~\ref{cor:flow} which interprets $\pasm(\nu / \lambda,m,n)$ as a flow polytope.
The first definition will be used to specify a graph $G_{P(\nu / \lambda)}$ corresponding to the poset $P(\nu / \lambda)$.

\begin{definition}
Given a poset $P$, let $\widehat{P}$ denote the poset constructed by adjoining a new minimum element $\widehat{0}$ and a new maximum element $\widehat{1}$. A poset $P$ is \emph{strongly planar} if the Hasse diagram of $\widehat{P}$ has a planar embedding into $\mathbb{R}^2$ such that if vertex $p_i$ is in position $(x_i, y_i)$ then $y_i < y_j$ (in the usual ordering on $\mathbb{R}$) whenever
$p_i < p_j$ in $P$. 

The \emph{truncated dual} of a planar graph $G$ is the dual graph of $G$ with the vertex corresponding to the infinite face deleted.

Given a strongly planar poset $P$, let $H$ be the planar graph obtained from the Hasse diagram of $\widehat{P}$ with two additional edges
from $\widehat{0}$ to $\widehat{1}$, one of which goes to the left of all the poset elements and another to the right. We then define the graph $G_P$ to be the truncated dual of $H$. The orientation of $G_P$ is inherited from the poset in the following way: if in the construction of the truncated dual, the edge $e$ of $G_P$ crosses the edge $p_i \rightarrow p_j$ where $p_i < p_j$ in $P$, then $p_j$ is on the left and $p_i$ is on the right as you traverse $e$. 
\end{definition}

The next definition will be used to specify a polytope corresponding to the graph $G_{P(\nu / \lambda)}$.

\begin{definition}
Let $G$ be a connected graph on the vertex set $\left[k\right] := \{1, 2, \ldots, k\}$ with edges directed from smaller to larger vertex. Denote by $\text{in}(e)$ the initial vertex of edge $e$ and $\text{fin}(e)$ the final vertex of edge $e$.

A \emph{flow} $fl$ of size one on $G$ 
is a function $fl:E(G) \rightarrow \mathbb{R}_{\geq 0}$ such that 

\begin{align*} 
\sum_{e\in E, \text{in}(e)=1} fl(e) = \sum_{e\in E, \text{fin}(e)=k} fl(e) &= 1,\\
\sum_{e\in E, \text{fin}(e)=i} fl(e) = \sum_{e\in E, \text{in}(e)=i} fl(e) &\text{  for } 2 \leq i \leq k-1. 
\end{align*}
The \emph{flow polytope} $\mathcal{F}_G$ associated to the graph $G$ is the set of all flows $fl : E(G) \rightarrow \mathbb{R}_{\geq 0}$ of size one on $G$.
\end{definition}

We now obtain the following, as a corollary of Theorem ~\ref{thm:int_equiv} and \cite[Theorem 3.14]{ASMCRY}. See Figure~\ref{fig:flow_polytope} for an example.

\begin{corollary}
\label{cor:flow}
The $(\lambda,n)$-partial ASM polytope $\pasm(\nu / \lambda,m,n)$ is integrally equivalent to the flow polytope $\mathcal{F}_{G_{P(\nu / \lambda)}}$.
\end{corollary}

\begin{figure}[hbtp]
\includegraphics[scale=.65]{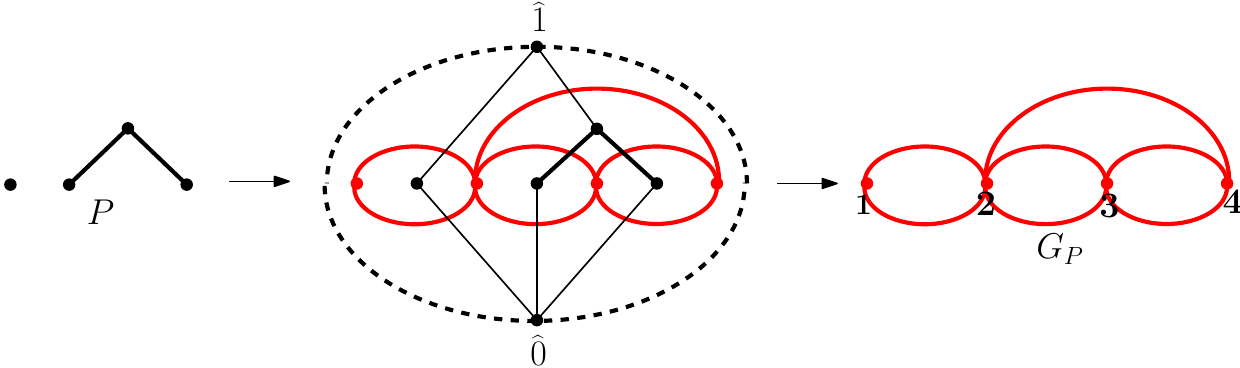}
\caption{Illustration of the map from the order polytope $\mathcal{O}(P(\nu / \lambda))$ to the flow polytope $\mathcal{F}_{G_{P(\nu / \lambda)}}$ for $\lambda = (3,1)$ and $\nu = (4,2,2)$}
\label{fig:flow_polytope}
\end{figure}

\begin{proof}
By Theorem~\ref{thm:int_equiv}, $\pasm(\nu / \lambda,m,n)$ is integrally equivalent to the order polytope $\mathcal{O}(P(\nu / \lambda))$. By \cite[Theorem 3.14]{ASMCRY}, since the poset $P(\nu / \lambda)$
is strongly planar, $\mathcal{O}(P(\nu / \lambda))$ is integrally
equivalent to the flow polytope $\mathcal{F}_{G_{P(\nu / \lambda)}}$. Thus, $\pasm(\nu / \lambda,m,n)$ is also integrally
equivalent to the flow polytope $\mathcal{F}_{G_{P(\nu / \lambda)}}$. 
\end{proof}

Our final result describes how $\pasm(\nu / \lambda,m,n)$ relates to the $\lambda$-ASMCRY polytope $\asmcry(\lambda,n)$ of \cite{ASMCRY} (see  Definition~\ref{def:asmcry}) in the case $m=n$ and $\nu = \delta_n := (n-1,n-2,\ldots,2,1)$.
\begin{theorem}
\label{prop:equiv}
Given $\lambda\subseteq(n-1,n-2,\ldots,2,1)\subset(n-1)^{n-1}$, an explicit integral equivalence $\varphi$ from $\pasm(\delta_n / \lambda,n,n)$ to  $\asmcry(\lambda,n)$ is given as follows. For each $M\in\pasm(\delta_n / \lambda,n,n)$, define 
\[\varphi(M)=
\begin{cases}
M_{ij}+1 & \text{if } j= n-i+2\\
M_{ij} & \text{otherwise.}
\end{cases}
\] 
\end{theorem}
\begin{proof}
By Theorem \ref{thm:int_equiv} and \cite[Theorem 1.1]{ASMCRY}, respectively, we have that $\pasm(\delta_n / \lambda,n,n)$ and $\asmcry(\lambda,n)$ are each integrally equivalent to the order polytope $\mathcal{O}(P(\delta_n / \lambda))$. 
The map $\varphi$ given above is a translation that takes vertices to vertices, as shown below, thus it is an integral equivalence between these two polytopes.
The reason it takes vertices to vertices is as follows. A vertex in $\pasm(\delta_n / \lambda,n,n)$ is a matrix $M$ with first row and column sums of 1 and the others $0$. A vertex in $\asmcry(\lambda,n)$ is a matrix with all row and column sums of 1. Adding 1 to the last nonfixed entry in rows 2 through $n$ of $M$ changes exactly those row and column sums which were 0 into 1, so \eqref{eq:ASM3} is satisfied. Also, since the entries below the main antidiagonal in $M$ are either $0$ or $-1$, adding $1$ to these gives $1$ or $0$. If such an entry in $M$ was $-1$, there was a $1$ somewhere to its left in that row and somewhere above it in the column. If the entry in $M$ was $0$, the entire row consisted of $0$ entries or there was a $-1$ somewhere to its left in that row. Similarly, either the entire column consisted of $0$ entries or there was a $-1$ somewhere above it in the column. So adding $1$ to the entries in all of these cases satisfies \eqref{eq:ASM1} and \eqref{eq:ASM2}. Thus this results in a vertex of $\asmcry(\lambda,n)$.
\end{proof}

This shows our work is a natural but non-obvious generalization of the results of \cite{ASMCRY} to the partial alternating sign matrix case.
See Figure~\ref{fig:pasmasmcry} for an example of this correspondence.

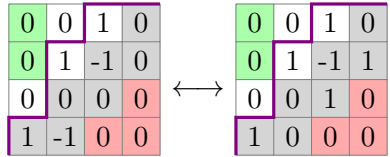
\begin{figure}[htbp]
\centering
\begin{tikzpicture}
\begin{scope}[scale=0.5]
    \draw[fill, color=green!33] (0,2) rectangle (1,4);
    
    \draw[fill, color=gray!33] (0,0) rectangle (2,1);
    \draw[fill, color=gray!33] (1,1) rectangle (3,2);
    \draw[fill, color=gray!33] (2,2) rectangle (4,3);
    \draw[fill, color=gray!33] (3,3) rectangle (4,4);
    
    \draw[fill, color=red!33] (2,0) rectangle (4,1);
    \draw[fill, color=red!33] (3,1) rectangle (4,2);
    
    \draw[very thin, color=gray!100] (0,0) grid (4,4);
    \draw[very thick, violet] (0,0) -- (0,1) -- (1,1) -- (1,3) -- (2,3) -- (2,4) -- (4,4) ;
      
    \node at (0.5,3.5) {0};
    \node at (1.5,3.5) {0};
    \node at (2.5,3.5) {1};
    \node at (3.5,3.5) {0};
   
    \node at (0.5,2.5) {0};
    \node at (1.5,2.5) {1};
    \node at (2.5,2.5) {-1};
    \node at (3.5,2.5) {0};
    
    \node at (0.5,1.5) {0};
    \node at (1.5,1.5) {0};
    \node at (2.5,1.5) {0};
    \node at (3.5,1.5) {0};
    
    \node at (0.5,0.5) {1};
    \node at (1.5,0.5) {-1};
    \node at (2.5,0.5) {0};
    \node at (3.5,0.5) {0};
    
\end{scope}
\end{tikzpicture}
\raisebox{0.8cm}{$\longleftrightarrow$}
\begin{tikzpicture}
\begin{scope}[scale=0.5]
    \draw[fill, color=green!33] (0,2) rectangle (1,4);    
    
    \draw[fill, color=gray!33] (0,0) rectangle (2,1);
    \draw[fill, color=gray!33] (1,1) rectangle (3,2);
    \draw[fill, color=gray!33] (2,2) rectangle (4,3);
    \draw[fill, color=gray!33] (3,3) rectangle (4,4);
    
    \draw[fill, color=red!33] (2,0) rectangle (4,1);
    \draw[fill, color=red!33] (3,1) rectangle (4,2);
    
    \draw[very thin, color=gray!100] (0,0) grid (4,4);
    \draw[very thick, violet] (0,0) -- (0,1) -- (1,1) -- (1,3) -- (2,3) -- (2,4) -- (4,4) ;
      
    \node at (0.5,3.5) {0};
    \node at (1.5,3.5) {0};
    \node at (2.5,3.5) {1};
    \node at (3.5,3.5) {0};
   
    \node at (0.5,2.5) {0};
    \node at (1.5,2.5) {1};
    \node at (2.5,2.5) {-1};
    \node at (3.5,2.5) {1};
    
    \node at (0.5,1.5) {0};
    \node at (1.5,1.5) {0};
    \node at (2.5,1.5) {1};
    \node at (3.5,1.5) {0};
    
    \node at (0.5,0.5) {1};
    \node at (1.5,0.5) {0};
    \node at (2.5,0.5) {0};
    \node at (3.5,0.5) {0};
    
\end{scope}
\end{tikzpicture}
\caption{An example of the correspondence in Theorem~\ref{prop:equiv}. Left: a matrix in $\pasm(\delta_4 / \lambda,4,4)$ with $\lambda = (1,1)$. Right: the corresponding matrix in $\asmcry(\lambda,4)$.}
\label{fig:pasmasmcry}
\end{figure}

\section*{Acknowledgments}
JS acknowledges support from the Simons Foundation gift MP-TSM-00002802 and NSF grant DMS-2247089.

\bibliographystyle{plain}
\bibliography{biblio}

\begin{thebibliography}{10}

\bibitem{Behrend2007HigherSA}
Roger~E. Behrend and Vincent~A. Knight.
\newblock Higher spin alternating sign matrices.
\newblock {\em Electron. J. Comb.}, 14, 2007.

\bibitem{BGHHKMY2019}
Carolina Benedetti, Rafael~S. Gonz\'{a}lez~D'Le\'{o}n, Christopher R.~H.
  Hanusa, Pamela~E. Harris, Apoorva Khare, Alejandro~H. Morales, and Martha
  Yip.
\newblock A combinatorial model for computing volumes of flow polytopes.
\newblock {\em Trans. Amer. Math. Soc.}, 372(5):3369--3404, 2019.

\bibitem{CRY2000}
Clara~S. Chan, David~P. Robbins, and David~S. Yuen.
\newblock On the volume of a certain polytope.
\newblock {\em Experiment. Math.}, 9(1):91--99, 2000.

\bibitem{HeuerStriker}
Dylan Heuer and Jessica Striker.
\newblock Partial permutation and alternating sign matrix polytopes.
\newblock {\em SIAM J. Discrete Math.}, 36(4):2863--2888, 2022.

\bibitem{JangKim2020}
Jihyeug Jang and Jang~Soo Kim.
\newblock Volumes of flow polytopes related to caracol graphs.
\newblock {\em Electron. J. Combin.}, 27(4):Paper No. 4.21, 21, 2020.

\bibitem{kuperbergASMpf}
Greg Kuperberg.
\newblock Another proof of the alternating-sign matrix conjecture.
\newblock {\em Internat. Math. Res. Notices}, (3):139--150, 1996.

\bibitem{LiuMeszarosStDizier2019}
Ricky~I. Liu, Karola M\'{e}sz\'{a}ros, and Avery St.~Dizier.
\newblock Gelfand-{T}setlin polytopes: a story of flow and order polytopes.
\newblock {\em SIAM J. Discrete Math.}, 33(4):2394--2415, 2019.

\bibitem{MeszarosMorales2019}
Karola M\'{e}sz\'{a}ros and Alejandro~H. Morales.
\newblock Volumes and {E}hrhart polynomials of flow polytopes.
\newblock {\em Math. Z.}, 293(3-4):1369--1401, 2019.

\bibitem{ASMCRY}
Karola M\'{e}sz\'{a}ros, Alejandro~H. Morales, and Jessica Striker.
\newblock On flow polytopes, order polytopes, and certain faces of the
  alternating sign matrix polytope.
\newblock {\em Discrete Comput. Geom.}, 62(1):128--163, 2019.

\bibitem{MRRASMDPP}
W.~H. Mills, David~P. Robbins, and Howard Rumsey, Jr.
\newblock Alternating sign matrices and descending plane partitions.
\newblock {\em J. Combin. Theory Ser. A}, 34(3):340--359, 1983.

\bibitem{MPP2018}
Alejandro~H. Morales, Igor Pak, and Greta Panova.
\newblock Hook formulas for skew shapes {I}. {$q$}-analogues and bijections.
\newblock {\em J. Combin. Theory Ser. A}, 154:350--405, 2018.

\bibitem{NaruseHLF}
H.\ Naruse.
\newblock Schubert calculus and hook formula, 2014.
\newblock https://www.mat.univie.ac.at/~slc/wpapers/s73vortrag/naruse.pdf.

\bibitem{RobbinsRumsey}
David~P. Robbins and Howard Rumsey, Jr.
\newblock Determinants and alternating sign matrices.
\newblock {\em Adv. in Math.}, 62(2):169--184, 1986.

\bibitem{SolhjemStriker}
Sara Solhjem and Jessica Striker.
\newblock Sign matrix polytopes from {Y}oung tableaux.
\newblock {\em Linear Algebra Appl.}, 574:84--122, 2019.

\bibitem{Stop}
Richard~P. Stanley.
\newblock Two poset polytopes.
\newblock {\em Discrete Comput. Geom.}, 1(1):9--23, 1986.

\bibitem{ASMPoly}
Jessica Striker.
\newblock The alternating sign matrix polytope.
\newblock {\em Electron. J. Comb.}, 16, 2009.

\bibitem{unifyingframework}
Matias von Bell, Rafael~S. Gonz\'{a}lez~D'Le\'{o}n, Francisco~A.
  Mayorga~Cetina, and Martha Yip.
\newblock A unifying framework for the {$\nu$}-{T}amari lattice and principal
  order ideals in {Y}oung's lattice.
\newblock {\em Combinatorica}, 43(3):479--504, 2023.

\bibitem{ZEILASM}
Doron Zeilberger.
\newblock Proof of the alternating sign matrix conjecture.
\newblock {\em Electron. J. Combin.}, 3(2):Research Paper 13, approx. 84, 1996.
\newblock The Foata Festschrift.

\bibitem{ZeilbergerCRY}
Doron Zeilberger.
\newblock Proof of a conjecture of {C}han, {R}obbins, and {Y}uen.
\newblock {\em Electron. Trans. Numer. Anal.}, 9:147--148, 1999.
\newblock Orthogonal polynomials: numerical and symbolic algorithms
  (Legan\'{e}s, 1998).

\end{thebibliography}
\end{document}